\documentclass[11pt]{article}

\input{WAAHmac2}  % macros updated for v2

\title{\vskip-1.0em\sc Weak amenability for Fourier algebras of $1$-connected nilpotent Lie groups}
\author{\sc Y.Choi and M. Ghandehari}

\date{1st March 2015}

\usepackage{graphics}
\usepackage{sectsty}
\allsectionsfont{\sffamily}

\usepackage[usenames]{color}
\newcommand{\dt}[1]{\textcolor{Bittersweet}{\sf#1}}
\newcommand{\para}[1]{\paragraph{#1.}}  % in elsarticle paragraph headings get automatic full stop

\begin{document}

\maketitle

\begin{abstract}
A special case of a conjecture raised by Forrest and Runde (Math.\ Zeit., 2005) asserts that the Fourier algebra of every non-abelian connected Lie group fails to be weakly amenable; this was already known to hold in the non-abelian compact cases, by earlier work of Johnson (JLMS, 1994) and Plymen (unpublished note).
In recent work (JFA, 2014) the present authors verified this conjecture for the real $ax+b$ group and hence, by structure theory, for any semisimple Lie group.

In this paper we verify the conjecture for all $1$-connected, non-abelian nilpotent Lie groups, by reducing the problem to the case of the Heisenberg group.
As in our previous paper, an explicit non-zero derivation is constructed on a dense subalgebra, and then shown to be bounded using harmonic analysis. {\it En route}\/ we use the known fusion rules for Schr\"odinger representations to give a concrete realization of the ``dual convolution'' for this group as a kind of twisted, operator-valued convolution. We also give some partial results for solvable groups which give further evidence to support the general conjecture.

 % abstract updated for v2
 % typo corrected for v2

\end{abstract}

%%%%%%%%%%%%%%%%%%
% MAIN CONTENT HERE
% THIS GETS ROUND HAVING DIFFERENT PREAMBLE MACROS FOR JOURNALS ETC

% Corrects some typos that were found by JFA copyeditors
% Does not include all the changes made or requested by JFA
% main body of document
\begin{section}{Introduction}
Fourier algebras of locally compact groups comprise an interesting class of Banach function algebras whose detailed structure remains somewhat mysterious, especially for groups which are neither compact nor abelian. 
It was observed by B.~E. Forrest \cite{Forr_PAMS88} that these algebras have no non-zero continuous point derivations. Nevertheless, as part of his seminal paper \cite{BEJ_AG}, B.~E. Johnson constructed a continuous non-zero derivation from the Fourier algebra of $\SO(3)$ into a suitable Banach bimodule: in the language of \cite{BCD_WA}, he proved that the Fourier algebra of $\SO(3)$ is not \dt{weakly amenable}.
This can be interpreted as evidence for some kind of weak form of differentiability or H\"older continuity for functions in the algebra.

Sufficient conditions for weak amenability were obtained in \cite{ForRun_amenAG}: if $G$ is locally compact and the connected component of its identity element is abelian, then its Fourier algebra $\FA(G)$ is weakly amenable. 
Motivated by Johnson's result, the authors of \cite{ForRun_amenAG} conjectured that this sufficient condition for weak amenability of $\FA(G)$ is necessary.
In particular, their conjecture implies that the Fourier algebra of any non-abelian connected Lie group is not weakly amenable; this was known at the time for compact Lie groups, but unknown for several natural examples including $\SL(2,\Real)$ and all the nilpotent cases.

This paper is a sequel to \cite{CG_WAAG1}, which studied weak and cyclic amenability for Fourier algebras of certain connected Lie groups, and whose introduction contains further information on the history and context of the results mentioned above.
In that paper we showed that the Fourier algebra of any connected, semisimple Lie group fails to be weakly amenable.
The key to this result was to show that the Fourier algebra of the real $ax+b$ group is not weakly amenable, and this in turn was done by 
constructing an explicit non-zero derivation from the Fourier algebra to its dual.
The derivation constructed in \cite{CG_WAAG1} is easily defined on a dense subalgebra, but showing that it extends continuously to the whole Fourier algebra required careful estimates provided by explicit orthogonality relations for certain coefficient functions of the real $ax+b$ group.
We also proved, using similar techniques, that the Fourier algebra of the reduced Heisenberg group $\bbH_r$ is not weakly amenable.
However, our methods were not able to handle the Fourier algebra of the ``full'' $3$-dimensional Heisenberg group~$\bbH$, which is a key example to consider when seeking to prove or refute the conjecture of Forrest and Runde. 

In the present paper we develop techniques which allow us to fill this gap. The outcome is the following new result.

\begin{thm}\label{t:mainthm}
There exist a symmetric Banach bimodule $\sW$ and a bounded, non-zero derivation $D:\FAH\to \sW$. Consequently, $\FAH$ is not weakly amenable.
\end{thm}

 This result then opens the way, via structure theory of Lie groups and Herz's restriction theorem for Fourier algebras, to the following more general statement.

\begin{thm}\label{t:nilpotent cases}
Let $G$ be a $1$-connected Lie group. If $G$ is also nilpotent and non-abelian, then $\FA(G)$ is not weakly amenable.
\end{thm}

(In the present context, $1$-connected is a synonym for ``connected and simply connected''; we are following the terminology of \cite{HN_LieBook}.)

\subsection*{Outline of our approach}
As in \cite{CG_WAAG1}, we construct an explicit non-zero derivation on a dense subalgebra, and then use harmonic analysis to show this derivation has a bounded extension to~$\FAH$.
 We follow the same informal guiding principle as before: use the Fourier transform to convert a claim about a differential operator to one about some kind of Fourier multiplier.
However, since $\bbH$ is far from being an AR group, we cannot use a decomposition of its Fourier algebra into coefficient spaces of square-integrable representations. (Contrast this with the ideas sketched in \cite[Section~7]{CG_WAAG1} for the group $\bbH_r$.)
We are therefore forced to use a different perspective: instead of orthogonality relations for coefficient functions, we use a version of the Plancherel formula for~$\bbH$.

We then encounter another obstacle not present in our previous paper.
In \cite{CG_WAAG1}, the derivation constructed for the $ax+b$ group mapped the Fourier algebra to its dual. We are unable to do the same for $\bbH$, but instead construct a derivation taking values in a Banach space $\sW$ that is constructed artificially for our purposes. The way we define $\sW$ makes it easy to show our derivation extends to a continuous linear map $\FAH\to \sW$, but the work lies in showing that $\sW$ is a genuine Banach $\FAH$-bimodule for the natural pointwise product.

In fact, to prove that the norm $\wnorm{\cdot}$ is an $\FAH$-module norm, we study the dual norm $\mnorm{\cdot}$ and prove that \emph{this} norm is an $\FAH$-module norm. This may seem unmotivated, but looking at the arguments of \cite{CG_WAAG1} for the \emph{reduced} Heisenberg group $\bbH_r$\/, one sees that there the problem is solved by
 establishing an estimate
\[ \left\vert \int_{\bbH_r} (\dd_Z f)(\bx)g(\bx)\,d\bx \right\vert \leq \norm{f}_{\FA(\bbH_r)} \norm{g}_{\FA(\bbH_r)} \]
where $\dd_Z$ is a certain normalized partial derivative. It is then not such a leap to look for an estimate of the form 
\[ \left\vert \int_{\bbH} (\dd_Z f)(\bx)g(\bx)\,d\bx \right\vert \leq \norm{f}_{\FA(\bbH)} \mnorm{g} \,, \]
provided we can show $\mnorm{\cdot}$ is an $\FAH$-module norm.

How do we prove $\mnorm{\cdot}$ is an $\FAH$-module norm? It turns out that this can be done very easily if we use the Fourier transform to move everything over to the ``Fourier side'', identifying $\FAH$ with a vector-valued $L^1$-space. We then need to study the product on this vector-valued $L^1$-space corresponding to pointwise product on $\FAH$;
this can be expressed as an explicit ``twisted operator-valued convolution'', and then the required inequalities for $\mnorm{\cdot}$ follow from standard properties of the Bochner integral.
The details of this operator-valued convolution (which is a concrete version of a general construction studied in e.g.~\cite[\S9]{Sti_TAMS59}) are given in Section~\ref{s:dualconv}.

 The special feature of $\bbH$ which makes this work is that the  fusion rules for the infinite-dimensional irreducible representations of $\bbH$ behave very nicely, so that the ``convolution'' has a very tractable form.
 We hope that these results may be of independent interest: a theme throughout this paper, also implicit in \cite{CG_WAAG1}, is that the group side is better for checking algebraic properties, such as the derivation identity and associativity of module actions, while the Fourier side is better for verifying norm estimates and approximating by well-behaved elements.

Finally, in Section \ref{s:moreLie} we return to Fourier algebras of more general Lie groups, and show how Theorem~\ref{t:nilpotent cases} follows from Theorem~\ref{t:mainthm}. We close with some further partial results and questions for Fourier algebras of solvable Lie groups.

\para{Note added in proof}
After this paper was submitted for publication, we learned of the interesting work of Lee--Ludwig--Samei--Spronk, \texttt{arXiv 1502.05214}, which proves the Lie case of the Forrest--Runde conjecture using a different perspective. In particular, the case of the motion group ${\rm Euc}(2)$, left open here, is resolved by these authors.

\subsection*{Acknowledgments}
This work was initiated while both authors worked at the University of Saskatchewan, where the first author was partially supported by NSERC Discovery Grant 4021530-2011 (Canada). The project was completed, and the first version of this paper written,  while both authors were visiting the Fields Institute for Mathematical Research, Toronto, as part of a thematic program on {\it Banach and Operator Algebras in Harmonic Analysis}, January--June 2014. We~thank the organizers of the program and the relevant concentration periods for the invitation to participate, and we thank the Fields Institute for their hospitality and support. The first author also thanks the Faculty of Science and Technology at Lancaster University, England, for fin\-ancial support to attend the thematic program.
The second author was supported by a Fields Postdoctoral Fellowship affiliated to the thematic program.

Important revisions were made, and some extra material added, while the second author was visiting Lancaster University under the support of a \emph{Scheme~2} grant from the London Mathematical Society. She thanks the Society for their support and the Department of Mathematics and Statistics at Lancaster University for its hospitality.

Finally, both authors wish to give special thanks to the referee of this article, for an attentive reading of the original submission, and for several precise and helpful suggestions which have improved the presentation of the original arguments. In particular we are grateful for the suggestion to look at \cite[\S9]{Sti_TAMS59} when discussing abstract Fourier inversion for unimodular groups, and for suggesting a simpler and more direct definition of the module~$\sW$ and the derivation $D:\FAH\to\sW$ in Theorem~\ref{t:W-is-a-module-norm}.
\end{section}

\begin{section}{Definitions and technical preliminaries}
\label{s:prelim}

To make the paper more self-contained, and to fix notation and establish terminology, we use this section to collect various definitions and statements from the literature.
% In an attempt to collate most of what we need in one place, we have chosen to repeat certain definitions and statements (without proof) from the literature. 
We thus hope to make the present paper more accessible to workers in the general area of Banach algebras, who may be less familiar with some of these technical preliminaries than specialists in harmonic analysis on Lie groups.
 The experienced reader may wish to skip this long preliminary section and go straight to Section~\ref{s:Heisenberg}; there, we specialize to the Heisenberg group, and present the Plancherel transform and Fourier inversion formula for this group in the form that we will need later.

\begin{subsection}{Notation and key definitions}

\begin{notn}[Banach spaces]
Throughout $\ptp$ will denote the projective tensor product of Banach spaces.
All Banach spaces are defined over complex scalars; if $E$ is a Banach space then $\overline{E}$ will denote the complex conjugate of $E$. It should be clear from context how to distinguish this from the usual notation for the closure of a set.

Given $p\in [1,\infty)$ and a Hilbert space $\cH$, we let $\cS_p(\cH)$ denote the space of \dt{$p$-Schatten class operators} on~$\cH$, and denote the corresponding $p$-Schatten norm by~$\norm{A}_p$. The operator norm on $\Bdd(\cH)$ will be denoted by $\norm{A}_\infty$.
\end{notn}

The following notation is standard and may be found in \cite{eymard64}, for instance.
\begin{dfn}[The ``check map'']
Let $G$ be a group and let $f:G\to\Cplx$ be an arbitrary function.
We denote by $\check{f}$ the function $g\mapsto f(g^{-1})$.
\end{dfn}

Given a Banach algebra $A$, a Banach $A$-bimodule $M$ is \dt{symmetric} if $a\cdot m = m\cdot a$ for all $a\in A$ and $m\in M$. A bounded linear map $D:A\to M$ is said to be a \dt{(continuous) derivation} if it satisfies the Leibniz identity $D(ab)=a\cdot D(b) + D(a)\cdot b$ for all $a,b\in A$.

The following definition is due to Bade, Curtis and Dales~\cite{BCD_WA}.
Let $A$ be a commutative Banach algebra.
 We say that $A$ is \dt{weakly amenable} if there is no non-zero, continuous derivation from $A$ to any symmetric Banach $A$-bimodule.

\begin{rem}\label{derivations on quotients}
As observed in \cite{BCD_WA}: if $A, B$ are commutative Banach algebras and $\theta:A\to B$ is a continuous homomorphism with dense range, then continuous, non-zero derivations on $B$ can be pulled back along $\theta$ to give continuous, \emph{non-zero} derivations on~$A$. Consequently, if $B$ fails to be weakly amenable, $A$ also fails to be weakly amenable.
\end{rem}

There are several different ways to define the Fourier algebra of a locally compact group, each with their own pros and cons. The following definition is not the original one, but is equivalent to it by the results of
\cite[Chapitre~3]{eymard64}.
Let $\lambda$ denote the left regular representation of $G$ on $L^2(G)$. Given $\xi,\eta\in L^2(G)$ we form the corresponding \dt{coefficient function} of $\lambda$\/, denoted by $\xi*_\lambda\eta$ and defined by
\[ (\xi*_\lambda\eta)(g) \defeq \pair{\lambda(g)\xi}{\eta} = \int_G \xi(g^{-1}s)\overline{\eta(s)}\,ds \,.\]
The map $\xi\tp\overline{\eta} \to \xi*_\lambda \eta$ defines a bounded linear map $\theta_\lambda: L^2(G)\ptp \overline{L^2(G)}\to C_0(G)$, and its range, equipped with the quotient norm of $L^2(G)\ptp \overline{L^2(G)}/\ker(\theta_\lambda)$, is denoted by $\FA(G)$. It follows from Fell's absorption theorem that $\FA(G)$ is closed under pointwise product and the norm on $\FA(G)$ is submultiplicative.
(See e.g.~\cite[\S4.1]{Zwarich_MSc} for a quick exposition of these results.)
 Thus $\FA(G)$ is a Banach algebra of functions on $G$, called the \dt{Fourier algebra} of $G$.
In fact, every element of $\FA(G)$ can be realized as a coefficient function of $\lambda$, and we have
\begin{equation}\label{eq:A(G)-as-A_lambda}
\norm{u}_{\FA(G)}=\inf\{\norm{\xi}_2\norm{\eta}_2 \st u=\xi*_\lambda\eta\}.
\end{equation}

If $H$ is a closed subgroup of $G$ let $\imath^*:C_0(G)\to C_0(H)$ be the restriction homomorphism. 
One can show that $\imath^*$ maps $\FA(G)$ contractively onto $\FA(H)$: this is originally due to C.~Herz, but an approach using spaces of coefficient functions was given by G. Arsac~\cite{Arsac}. (A fairly self-contained account of Arsac's approach can be found in \cite[\S4]{Zwarich_MSc}.)
Therefore, recalling Remark~\ref{derivations on quotients}, we obtain the following well-known result.

\begin{prop}\label{p:WA-hered}
Let $G$ be a locally compact group and $H$ a closed subgroup. If $\FA(H)$ is not weakly amenable, then $\FA(G)$ is not weakly amenable.
\end{prop}

\end{subsection}

\begin{subsection}{The Plancherel and inverse Fourier transforms for Type I unimodular groups}\label{ss:general Plancherel}
The Plancherel and (inverse) Fourier transform for the Heisenberg group will be important tools in our calculations.
These mappings can be defined in much more general settings: ``global versions'', valid for any locally compact unimodular group, can be found in work of Stinespring \cite[\S9]{Sti_TAMS59}.
However, for the key work in this paper (Section~\ref{s:dualconv}) it seems important to use concrete knowledge of the unitary dual and Plancherel measure for the Heisenberg group.

\begin{rem}[Minor caveat]\label{r:stinespring-caveat}
In more modern language, \cite[\S9]{Sti_TAMS59} works with noncommutative $L^p$-spaces of the pair $(\VN(G),\tau)$, where $\tau$ is the Plancherel weight for $\VN(G)$; these are, strictly speaking, certain spaces of $\tau$-measurable operators on $L^2(G)$. Later in this section, when we use results from Stinespring's paper to justify certain assertions, we are tacitly inserting an extra step: namely, one has to disintegrate $\VN(G)$ as a direct integral over $\widehat{G}$, and then observe that one can identify $L^p(\VN(G),\tau)$ with the corresponding space of $p$-integrable, $p$-Schatten class-valued operator fields over $\widehat{G}$.
\end{rem}

In future work, we intend to study similar problems concerning derivations for the Fourier algebras of some other Type~I groups, and it seems useful to collect some machinery here that is applicable to these other cases and not just to the Heisenberg group. Thus, for this subsection $G$ will be a second countable unimodular Type~I group.
We follow the terminology and definitions used in~\cite{Fuehr_LNM}. 

For such $G$, the canonical ``Mackey Borel structure'' on $\widehat{G}$ makes it into a standard measure space. Moreover, from each equivalence class of irreducible unitary representations, one can select a representative, in a way that gives a measurable field $(\cH_\pi)_{\pi\in\widehat{G}}$ and a corresponding measurable field of representations.
(See \cite[\S7.4]{Foll_AHAbook} for the basic definitions and properties of direct integrals and measurable fields, in particular Lemma 7.39 and Theorem 7.40 for the relevance of the Type~I condition.)

Using the notation of \cite{Fuehr_LNM}, Chapters 3 and~4: given a measure $\nu$ on $\widehat{G}$ and $1\leq p<\infty$, we write $\cB^\oplus_p(\widehat{G},\nu)$ for the space of all measurable fields $(T_\omega)$ which satisfy $T_\omega\in \cS_p(\cH_\omega)$ for $\nu$-a.e.~$\omega$ and
\[ \int_{\widehat{G}} \norm{T_\omega}_p^p \,d\nu(\omega) < \infty\,. \]
Once we make the usual identifications modulo $\nu$-a.e.\ equivalence,
$\cB^\oplus_p(\widehat{G},\nu)$ is a Banach space when equipped with the obvious norm $\norm{\cdot}_p$.

% We can now state the key properties of the Plancherel transform.

\para{Fact} (See \cite[Theorem 7.44]{Foll_AHAbook} or \cite[Theorem 3.31]{Fuehr_LNM}.)
There exists a measure $\nu$ on $\widehat{G}$, called the \dt{Plancherel measure of $G$}, such that the linear map
\[  \cP: f\mapsto (\pi(f))_{\pi\in\widehat{G}} \qquad(f\in (L^1\cap L^2)(G)) \]
takes values in $\cB^\oplus_2(\widehat{G},\nu)$ and satisfies $\norm{\cP(f)}_2=\norm{f}_{L^2(G)}$.
Moreover, if $f\in L^1(G)$ and $\pi(f)=0$ for $\nu$-a.e.\ $\pi\in\widehat{G}$, then $f=0$ a.e.\ on~$G$.

\begin{dfn}[Plancherel transform]
The map $\cP$ extends uniquely to a unitary isomorphism from $L^2(G)$ onto $\cB^\oplus_2(\widehat{G},\nu)$. This unitary isomorphism, which we also denote by $\cP$, is called the \dt{Plancherel transform} of $G$.
\end{dfn}

\begin{dfn}[Inverse Fourier transform]\label{def:inverseFourier}
With $G$ and $\nu$ as above, we define a bounded linear map
$\Psi: \cB^\oplus_1(\widehat{G},\nu)\to C_b(G)$ by
\begin{equation}
\Psi(F)(x)  \defeq \int_{\widehat{G}} \Tr (F(\pi)\pi(x)^*)\,d\nu(\pi).
\end{equation}
\end{dfn}

The following result is crucial to our calculations, since (for the particular case of the Heisenberg group) it allows us to work with a vector-valued $L^1$-norm rather than the norm of the Fourier algebra.

\begin{thm}[Arsac]\label{thm:ArsacInverseFourier}
The map $\Psi$ takes values in $\FA(G)$, and is an isometric isomorphism of Banach spaces from
$\cB^{\oplus}_1({\widehat{G}},\nu)$ onto $\FA(G)$.
\end{thm}

A proof is given in \cite[Theorem 4.12(a)]{Fuehr_LNM}; see also 
Propositions 3.53 and 3.55 of \cite{Arsac}, with the warning that our map $\Psi$ differs from Arsac's by an application of the ``check map'' $u\mapsto \check{u}$. In both proofs one uses the fact that
$\cP$ gives a unitary equivalence of representations $\lambda_G\simeq \int^\oplus_{\pi\in\widehat{G}}\pi\otimes I_\pi \, d\nu(\pi)$
% (see \cite[Theorem 7.44]{Foll_AHAbook}).
(see \cite[Theorem 3.31]{Fuehr_LNM}.)

Informally, we wish to say that $\Psi$ and $\cP$ are mutually inverse maps. However, to make this claim meaningful one must be more precise about what domain and codomain are being used for each map. The following theorem is enough to ensure that we do not run into difficulties.

\begin{thm}[Fourier inversion]\label{t:StinespringInverseFourier}
Let $F\in \cB^\oplus_1(\widehat{G},\nu)$. Suppose $\Psi(F)\in L^1(G)$. Then $\Psi(F)\in (L^1\cap L^2)(G)$ and
\[ \cP\Psi(F)=F\in \cB^\oplus_1(\widehat{G},\nu)\cap \cB^\oplus_2(\widehat{G},\nu). \]
\end{thm}

\begin{proof}
With appropriate changes of notation, and with Remark~\ref{r:stinespring-caveat} kept in mind, this is a special case of \cite[Theorem 9.17]{Sti_TAMS59}.
A much lengthier justification, which uses less theory of von Neumann algebras and measurable operators than \cite{Sti_TAMS59} but is more opaque, is given in the proof of \cite[Theorem 4.15]{Fuehr_LNM}.
\end{proof}

% From this we get the following result, which is vital at one point later on.
\begin{cor}\label{c:harder than you think}
Let $f\in (\FA\cap L^1)(G)$. Then $f\in (L^1\cap L^2)(G)$, $\cP(f)\in \cB^\oplus_1(\widehat{G},\nu)$ and $\norm{\cP(f)}_1 = \norm{f}_{\FA(G)}$.
\end{cor}

\begin{proof}
Let $f\in(\FA\cap L^1)(G)$.
By Theorem~\ref{thm:ArsacInverseFourier} there exists $F\in\cB^\oplus_1(\widehat{G})$ such that $\Psi(F)=f$ and $\norm{F}_1=\norm{f}_{\FA(G)}$.
By Theorem~\ref{t:StinespringInverseFourier}, $\cP(f)(\omega)=F(\omega)$ for $\nu$-a.e. $\omega\in\widehat{G}$, and the result follows.
\end{proof}

For certain~$G$ --- in particular, for the Heisenberg group $\bbH$ --- 
one can find a $\nu$-conull subset $\conull\subseteq\widehat{G}$ and choose a single Hilbert space $\cH$ on which to represent all $\pi\in\conull$.
This leads to a great simplification in the ``dual'' descriptions of $\FA(G)$ and $L^2(G)$, since the spaces $\cB^\oplus_p(\widehat{G},\nu)$ now take the form of operator-valued $L^p$-spaces. Since the ``dual'' picture of $\FAH$ is key to everything we do in this paper, we use the next subsection to collate some basic facts on vector-valued $L^p$-spaces and issues of measurability, which will be needed in the later sections.
\end{subsection}

\begin{subsection}{Vector-valued $L^p$-spaces, and the Bochner integral}
\label{ss:vector valued}

The basics of measurability and the Bochner integral can be found in \cite[Chapter 2]{DiestelUhl_book}. Strictly speaking, \cite{DiestelUhl_book} works exclusively with \emph{finite} measure spaces, but everything we need can be extended from the finite to the $\sigma$-finite setting in a straightforward way. In any case, the only places where we make serious use of precise properties of the Bochner integral is in Section~\ref{s:dualconv}, and there our measure spaces will be either $\Real$ or $\Real^2$ with usual Lebesgue measure; for an alternative reference, which only discusses those two measure spaces, see~\cite[Chapter 1]{ABHN_book}.

Let $X$ be a Banach space, let $(\Omega,\mu)$ be a $\sigma$-finite measure-space.
A \dt{simple function} $\Omega\to X$ is one of the form $\sum_{i=1}^m {\bf 1}_{E_i} x_i$ where $E_1,\dots, E_m$ are measurable subsets of $\Omega$ and $x_1,\dots,x_m\in~X$.
A function $F:\Omega\to X$ is \dt{(strongly)} measurable if it is the pointwise limit of a sequence of simple functions; it is called \dt{weakly measurable} if for each $\psi\in X^*$, the function $\psi\circ F:\Omega\to\Cplx$ is measurable.
If $X$ is a dual Banach space with predual $X_*$\/, a function $F:\Omega\to X$ is called \dt{weak-star measurable} if $\phi\circ F:\Omega\to\Real$ is measurable for each $\phi\in X_*$.

\begin{rem}\label{r:SOT-implies-wsmeas}
Let $\cK$ be a separable Hilbert space and recall that $\Bdd(\cK)_*=\cS_1(\cK)$. Hence, weak-star measurability of $F:\Omega\to\Bdd(\cK)$ is equivalent to measurability of each ``coefficient function'' $\omega\mapsto \ip{ F(\omega)\xi}{\eta}$ for every $\xi,\eta\in \cK$ (one direction is trivial and the other follows by taking limits of linear combinations of rank-one operators). In particular, every WOT-continuous function $\Omega\to \Bdd(\cK)$ is weak-star measurable.
\end{rem}

%\begin{thing}[one direction of the Pettis measurability theorem]
%\label{pettis theorem}
%Theorem. (See Theorem 2 in \cite[\S II.1]{DiestelUhl_book} or Theorem 1.1.1 in \cite{ABHN_book}.)
%Suppose $X$ is a separable Banach space, and $F:\Omega\to X$ is a weakly measurable function. Then $F$ is measurable. 
%\end{thing}

In Section~\ref{s:dualconv} we will need to use some standard properties of the Bochner integral.
The definition of Bochner integrability can be found in \cite[Chapter 1]{ABHN_book}. We will use the following characterization (see \cite[Theorem 1.1.4]{ABHN_book}):
a function $F:\Omega\to X$ is Bochner integrable if and only if it satisfies the following two conditions: (1) $F$ is measurable (2) the function $\omega\to\norm{F(\omega)}$ is integrable.

\begin{dfn}[Vector-valued $L^p$-spaces]
Let $(\Omega,\mu)$ be a $\sigma$-finite measure space and let $X$ be a Banach space.
For $1\leq p < \infty$, we define $L^p(\Omega,X)$ to be the space of all measurable functions $F:\Omega\to X$ which satisfy $\int_\Omega \norm{F(\omega)}^p \, d\mu <\infty$, modulo identification of functions that only differ on $\mu$-null sets. This is a Banach space for the norm
\[ \norm{F}_{L^p(\Omega,X)} \defeq \left( \int_{\Omega} \norm{F(\omega)}^p \, d\mu(\omega) \right)^{1/p} \]
$L^\infty(\Omega,X)$ denotes the space of essentially bounded, measurable functions $\Omega\to X$,
 modulo identification of functions that only differ on $\mu$-null sets.
This is a Banach space for the norm
\[ \norm{F}_{L^\infty(\Omega,X)} \defeq  \esssup_{\omega\in\Omega}\norm{F(\omega)}. \]
\end{dfn}

\end{subsection}

\end{section}

\begin{section}{Representations and Plancherel measure for the Heisenberg group}
\label{s:Heisenberg}
We now specialize to the case of the Heisenberg group, where the general results of Subsection~\ref{ss:general Plancherel} can be made much more concrete. None of the results in this section are new, but they are stated here for sake of consistency of terminology and notation.

\begin{dfn}[Heisenberg group]
We define the (real, $3$-dimensional) \dt{Heisenberg group} $\bbH$ to be the set $\Real^3$, equipped with its usual smooth manifold structure and equipped with the multiplication rule
\begin{equation}\label{eq:unpolarized}
(a_1,b_1,c_1) (a_2,b_2,c_2) \defeq (a_1+a_2,b_1+b_2, \frac{1}{2}(a_1b_2-a_2b_1)+c_1+c_2).
\end{equation}
This makes $\bbH$ into a connected, simply-connected, nilpotent Lie group; the centre of $\bbH$ is the subgroup $\{(0,0,c)\st c\in\Real\}$.
\end{dfn}

Note that this is sometimes called the \dt{symmetrized} or \dt{unpolarized} form of the Heisenberg group.
The \dt{polarized} form of the Heisenberg group may be described as:
\[ \Hpol\defeq \left\{ \left( \begin{matrix} 1 & a & c \\ 0 & 1 & b \\ 0 & 0 & 1 \end{matrix}\right) \colon a,b,c\in\Real\right\} \subset {\rm GL}(3,\Real).\]
$\bbH$ and $\Hpol$ are isomorphic as Lie groups, so one can easily convert results and definitions for one into results and definitions for the other.
Our choice to work with $\bbH$ rather than $\Hpol$, as well as the form of the next definition, follows \cite[\S7.6]{Foll_AHAbook}.

\begin{dfn}[Schr\"odinger representations]
Let $t\in\Real^*$. There is a continuous unitary representation $\pi_t:\bbH\to\cU(L^2(\Real))$ given by
\begin{equation}\label{eq:define schroedinger}
\pi_t(x,y,z)f(w)=e^{2\pi i tz+\pi i tyx}e^{-2\pi i tyw}f(w-x).
\end{equation}
It can be shown that $\pi_t$ is irreducible. We call it the \dt{Schr\"odinger representation} indexed by~$t$.
\end{dfn}

\begin{rem}[$\pi_t$ as an induced representation]\label{r:sch-as-induced}
$\bbH$ is isomorphic as a Lie group to a certain semidirect product $\Real^2\rtimes_{\alpha}\Real$ (this is slightly easier to see if one works with the polarized form $\Hpol$, see e.g.~Example 4.10 in~\cite{KaniuthTaylor}).
% start correction
When we induce $1$-dimensional representations of $\Real^2$ up to $\Real^2\rtimes\Real$, Mackey theory shows us that some of these induced representations are irreducible, and in fact they are unitarily equivalent to the Schr\"odinger representations. See Example 4.38 in~\cite{KaniuthTaylor} for further details.
% end correction
Although we will not explicitly use this perspective during the present paper, it serves to explain the ``fusion rules'' for the Schr\"odinger representations, which we \emph{do} need. We will return to this in Remark~\ref{r:out-of-hat}.
\end{rem}

Although $\widehat{\bbH}$ is not Hausdorff, it can be shown that the subset
$\conull\defeq \{ \pi_t \st t\in\Real^*\}$ is dense in $\widehat{\bbH}$ and is homeomorphic in the subspace topology to $\Real^*$: moreover, $\conull$ is  conull for the Plancherel measure~$\nu$, and one can identify the measure space $(\conull,\nu)$ with the measure space $(\Real^*, |t|dt)$. 
(For a fairly self-contained proof, see \cite[\S7.6]{Foll_AHAbook}.)
%Note that all the continuous irreducible unitary representations which appear in $\conull$ are defined on the same Hilbert space, namely $L^2(\Real)$.

\begin{notn}
For $p\in [1,\infty)$ we abbreviate $\cS_p(L^2(\Real))$ to $\STN_p$; we write $\sK$ for $\cK(L^2(\Real))$ and $\sB$ for $\Bdd(L^2(\Real))$.
\end{notn}

The spaces $\cB^\oplus_p(\widehat{\bbH},\nu)$
% from the previous section 
admit a much simpler description: we can identify $\cB^\oplus_p(\widehat{\bbH},\nu)$ with the vector-valued $L^p$-space $L^p(\conull,\STN_p)$, in the sense of Subsection~\ref{ss:vector valued}. In particular, the space $\cB^\oplus_1(\widehat{\bbH},\nu)$, which by Theorem~\ref{thm:ArsacInverseFourier} 
is isometrically isomorphic to the Fourier algebra $\FAH$, is nothing but the space of Bochner integrable $\STN_1$-valued functions on the measure space $(\Real^*, |t|dt)$.

\para{A change of measure}
To simplify some formulas in Sections \ref{s:define derivation}
and \ref{s:dualconv}, it is convenient not to work on $(\conull,\nu)$, nor on $(\Real^*,|t|dt)$, but on $\Real$ equipped with usual Lebesgue measure.
So we introduce $\cT: L^1(\Real,\STN_1)\to L^1(\conull,\STN_1)$ given by $\cT(F)(t)=|t|^{-1}F(t)$\/: this is an isometric isomorphism of Banach spaces.
Note that $\Psi\cT$ is given explicitly by
\begin{equation}\label{eq:tweaked}
\Psi\cT(F)(\bx) = \int_{\Real} \Tr[ F(t)\pi_t(\bx)^*]\,dt \qquad(F\in\LRS).
\end{equation}

Next, for sake of clarity, we state some results mentioned in Subsection~\ref{ss:general Plancherel} in the particular form that we need for later sections.

\vfill\eject

\begin{thm}[Fourier inversion for $\bbH$]
\label{t:F-inversion-for-AH} \

\begin{newnum}
\item\label{li:Psi T maps onto A(H)}
 $\Psi\cT$, as defined in \eqref{eq:tweaked}, maps $L^1(\Real,\STN_1)$ isometrically onto~$\FAH$.

\item
\label{li:hard bit}
 Let $f\in (\FA\cap L^1)(\bbH)$. Then the function $t\mapsto |t|\pi_t(f)$ belongs to $\LRS$, and coincides with $(\Psi\cT)^{-1}(f)$. In particular
\[ \int_{\Real} |t|\norm{\pi_t(f)}_1\,dt = \norm{f}_{\FAH}\,.\]

\end{newnum}
\end{thm}

The following identity, which will be useful at one point in proving Theorem~\ref{t:W-is-a-module-norm}, is a special case of \cite[Theorem 9.6]{Sti_TAMS59}.

\begin{lem}[Adjoint relation]
\label{l:adjoint-relation}
Let $g\in L^1(\bbH)$ and let $F\in L^1(\Real,\STN_1)$. Then
\begin{equation}\label{eq:adjoint}
\int_{\bbH} g(\bx) \Psi\cT(F)(\bx) \,d\bx = \int_{\Real^*} \Tr(\pi_t(\check{g})F(t))\,dt\,.
\end{equation}
\end{lem}

\begin{rem}
We used \cite{Sti_TAMS59} as a convenient reference here. It may be worth noting that Lemma~\ref{l:adjoint-relation} does not really rely on any form of the Plancherel formula or Fourier inversion: indeed, it can be proved directly from the definition of $\Psi\cT$, by using the Fubini--Tonelli theorem to evaluate
\[ \int_{\bbH\times\Real} g(\bx)\Tr(F(t)\pi_t(\bx)^*)\,d(\bx,t)\]
as an iterated integral in two different ways.
\end{rem}

\end{section}
\begin{section}{Defining our derivation}
\label{s:define derivation}
Let $\cC= (\FA\cap C^1_c)(\bbH)$: clearly this is an algebra with respect to pointwise product. 
Moreover, since $f*_\lambda g\in\cC$ whenever $f,g\in C^1_c(\bbH)$, $\cC$ is a dense subalgebra of $\FA(\bbH)$.
Let  $\dd_Z :\cC\to C_c(\bbH)$ be defined by
\begin{equation}
\dd_Z f(x,y,z) = - \frac{1}{2\pi i} \dbydz{f}(x,y,z) .
\end{equation}
Key to our approach is the fact that the ``Fourier multiplier'' corresponding to $\dd_Z$ is very simple.

\begin{lem}\label{lem:d-hat}
Let $f\in \cC$. Then 
 $\pi_t(\dd_Z f) =  t \pi_t(f)$ for all $t\in\Real^*$.
\end{lem}

\begin{proof}
Observe that
\[ \begin{aligned}
2\pi i\pi_t(\dd_Z f)
 = - \int_{\bbH} \dbydz{f}(\bx) \pi_t(\bx) \,d\bx 
 = \int_{\bbH} f(\bx) \frac{\partial}{\partial z}\pi_t(\bx) \,d\bx \,.
\end{aligned} \]
(There are no issues with differentiating under the integral sign or integrating by parts, since $f$ and $\dbydz{f}$ are continuous with compact support.) A straightforward calculation (see \eqref{eq:define schroedinger}) shows that $\frac{\partial}{\partial z}\pi_t(\bx) = 2\pi i t \pi_t(\bx)$,
and the rest is clear.
\end{proof}

We now specify the target space for our supposed derivation.
Given $f\in L^1(\bbH)$, define\hfill\break
\[ \wnorm{f} \defeq \int_{\Real} \opnorm{\pi_t(f)} \,dt \in [0,\infty] \,,\]
and then let $\sW_0\defeq \{f\in L^1(\bbH) \st \wnorm{f} < \infty\}$. Note that on $\sW_0$, $\wnorm{\cdot}$ is not just a seminorm, but a genuine norm
 (see the comments in Section~\ref{ss:general Plancherel}).
 Finally, let $\sW$ be the completion of $\sW_0$ with respect to the norm $\wnorm{\cdot}$. (Morally speaking, we think of $\sW$ as consisting of certain distributions on $\bbH$ whose Fourier transforms belong to $\LRB$.)

%\begin{rem}
%The restriction to $(L^1\cap L^2)(\bbH)$ is just for technical convenience and ensures that for each $f\in W_0$, the function $t\mapsto \pi_t(f)$ belongs to $\LRK$. Actually, for any $f\in L^1(\bbH)$ we know that $\pi_t(f)\in\sK$ for every $t\in\Real^*$, because the Heisenberg group is CCR/liminal; but it seemed more in the spirit of this paper to use a softer reason that might generalize more easily to other cases.
%\end{rem}

% Note that $L^1(\bbH)$ is a $C_0(\bbH)$-module (for pointwise product), but it is far from clear if $W_0$ is a sub-$C_0(\bbH)$-module. 

\begin{lem}\label{l:bounded}
If $f\in\cC$ then $\dd_Z f\in\sW_0$, and $\wnorm{\dd_Z f}\leq \norm{f}_{\FAH}$.
\end{lem}

\begin{proof}
First note that $\dd_Z f \in C_c(\bbH)\subset L^1(\bbH)$. By Lemma~\ref{lem:d-hat},
\[ \wnorm{\dd_Z f} = \int_{\Real} \opnorm{t\pi_t(f)}\,dt \leq \int_{\Real} |t| \norm{\pi_t(f)}_1 \,dt \,; \]
and since $f\in (\FA\cap L^1)(\bbH)$, applying Theorem~\ref{t:F-inversion-for-AH} completes the proof.
\end{proof}

\begin{thm}[Creating a target module]
\label{t:W-is-a-module-norm}
We have
\begin{equation}\label{eq:W-module-norm}
 \wnorm{fh} \leq \norm{f}_{\FAH} \wnorm{h} \qquad\text{for all $f\in \FA(\bbH)$ and all $h\in \sW_0$\/.} \end{equation}
Consequently:
\begin{newnum}
\item $\sW_0$ is a sub-$\FA(\bbH)$-module of $L^1(\bbH)$, for pointwise product;
\item $\sW$ becomes a Banach $\FA(\bbH)$-module in a way that continuously extends the $\FA(\bbH)$-action on $\sW_0$.
\end{newnum}
\end{thm}

Let us assume for now that the theorem holds, and show how it implies $\FAH$ is not weakly amenable (as claimed in Theorem~\ref{t:mainthm}).

\begin{proof}[Proof that $\FAH$ is not weakly amenable]
Since $\sW_0$ is an $\FAH$-module for pointwise product it is certainly a $\cC$-module. It is immediate from the product rule that $D_0:\cC\to \sW_0$, $f\mapsto \dd_Z(f)$, is a derivation from $\cC$ to a $\cC$-module. Moreover, by Lemma~\ref{l:bounded}, $D_0$ is continuous if we equip $\cC$ with the $\FAH$-norm and $\sW_0$ with the norm $\wnorm{\cdot}$.

By routine continuity arguments, since $\cC$ is dense in $\FAH$, there is a unique continuous linear map $D:\FAH\to \sW$ that extends $D_0$, and moreover $D$ is a derivation. It is not identically zero: for
 if $f\in\cC$ is a non-zero function, then $\dd_Z f$ is a non-zero element of $\sW_0$. Thus Theorem~\ref{t:mainthm} is proved.
\end{proof}

Our proof of Theorem~\ref{t:W-is-a-module-norm} is indirect. We shall study a norm which is dual to $\wnorm{\cdot}$, and show that this dual norm has the appropriate module property.
Then by a duality argument we will deduce the inequality \eqref{eq:W-module-norm}, after which the rest of the theorem follows easily.
% We note that in an earlier version of this paper, the module denoted here by $\sW$ was obtained as a norm-closed submodule of the dual of another module~$\sM$; the approach taken in the current version of this paper retains some vestige of the original philosophy.
Since the technical details of this part may obscure what is actually a natural and routine strategy, let us explain the underlying heuristics.

\para{Heuristics for our duality argument}
 Given $f\in \FAH$ and $h\in L^1({\mathbb H})$ and $g$ a well-behaved test function,  consider $\int_\bbH hfg \,d\bx$.
 The Plancherel theorem/Parseval formula tells us this is equal to
\begin{subequations}
\begin{equation}\label{eq:hf}
\int_{\Real} \Tr(\pi_t(hf)\pi_t(\overline{g})^*)\,|t|dt\,,
\end{equation}
and also equal to
\begin{equation}\label{eq:fg}
\int_{\Real} \Tr(\pi_t(h)\pi_t(\overline{fg})^*)\,|t|dt\,.
\end{equation}
\end{subequations}
For $u$ well-behaved, define $\mnorm{u}\defeq \esssup_{t\in\Real} |t|\norm{\pi_t(\overline{u})}_1$\/. Assume for the moment that $fg$ is also well-behaved; then \eqref{eq:hf} is bounded above by $\wnorm{hf}\mnorm{g}$ while \eqref{eq:fg} is bounded above by $\wnorm{h}\mnorm{fg}$. Now suppose we can prove the following two claims:
\begin{newnum}
\item\label{li:enough-to-test} if we take the supremum in \eqref{eq:hf} over all well-behaved test functions $g$ with $\mnorm{g}\leq 1$, we obtain the upper bound $\wnorm{hf}$;
\item\label{li:mnorm-is-dualconv}
 the new norm $\mnorm{\cdot}$ is a contractive $\FAH$-module norm, that is, $\mnorm{fg}\leq \norm{f}_{\FAH}\mnorm{g}$ for all well-behaved test functions $g$.
\end{newnum}
Then combining these two claims with the preceding remarks, we would obtain $\wnorm{hf}\leq \wnorm{h}\norm{f}_{\FAH}$ as required.

\begin{rem}
In the actual proof of Theorem~\ref{t:W-is-a-module-norm}, we do not define ``well-behaved'' test functions as elements of $L^1(\bbH)$ for which $\mnorm{\cdot}$ is finite, because of certain technical irritations that arise in verifying~\ref{li:enough-to-test}. It is more convenient, although perhaps less transparent, to take our test functions to be those of the form $\Psi\cT(G)$ where $G\in \LIRS$ has compact support, and to do norm calculations and dual pairings over on the Fourier side.
\end{rem}

The proof that $\mnorm{\cdot}$ is an $\FAH$-module norm turns out to be a very easy consequence of more general results, which describe explicitly the so-called ``dual convolution'' on $\LRS$ that corresponds to pointwise product in $\FAH$. This will be the main topic of the next section, where we will also finish the proof of Theorem~\ref{t:W-is-a-module-norm}.
\end{section}

\begin{section}{An explicit operator-valued convolution on $\LRS$}
\label{s:dualconv}
Since $\Psi\cT :\LRS\to\FAH$ is an isometric isomorphism, we may transport pointwise product on $\FAH$ over to define a commutative and associative multiplication map on $\LRS$, which we denote by $\opconv$. More precisely, given $F,G\in\LRS$ define
\[ F\opconv G = (\Psi\cT)^{-1} \left[ \Psi\cT(F)\Psi\cT(G) \right]\,.\]

In this abstract form, $\opconv$ is not new: it coincides -- modulo Remark~\ref{r:stinespring-caveat} -- with what is sometimes called ``dual convolution'' on the noncommutative $L^1$-space associated to the von Neumann algebra of a unimodular group, cf.~the definition on p.~48 of \cite[\S9]{Sti_TAMS59}. However, this abstract perspective does not seem helpful for proving that $\mnorm{\cdot}$ is an $\FAH$-module norm.
Instead, most of this section will be spent carefully deriving an explicit description of $\opconv$ as a kind of twisted convolution of operator-valued fields: see Equation \eqref{eq:explicit} and Theorem~\ref{t:explicit} below.
Once we have proved Theorem~\ref{t:explicit}, the desired inequalities will follow immediately from standard properties of the Bochner integral; and then the corresponding result for the dual norm $\wnorm{\cdot}$ will follow by a duality argument as sketched at the end of Section~\ref{s:define derivation}.

Since we hope that this ``concrete'' description of $\opconv$ may have independent interest, we treat the construction in some detail.
The starting point is the following loose idea: given $F,G\in \LRS$ we have
\[ \begin{aligned}
 \int_\Real \Tr( (F\opconv G)(t) \pi_t(\bx)^*) \, dt
& = \int_\Real\int_\Real \Tr(F(s)\pi_r(x)^*)\Tr(G(t)\pi_s(\bx)^*)\, d(r,s) \\
& = \int_\Real\int_\Real \Tr\left[ (F(r)\tp G(s) )(\pi_r\tp \pi_s)(\bx)^*\right] \, d(r,s)\,. \\
\end{aligned} \]
We then argue as follows: decompose or rewrite $\bx\mapsto \pi_r(\bx)\tp\pi_s(\bx)$ in terms of irreducible representations; rewrite the expression on the right-hand side of the formula above as a ``Fourier expansion'', and then appeal to uniqueness of Fourier coefficients in this expansion to get a reasonably explicit formula for $F\opconv G$.

%\begin{thing}[Remark to myself]
%It is tempting to try and use the identification $L^1(\Omega,X)\iso L^1(\Omega)\ptp X$ to simplify our arguments, with the idea that one can reduce to checking things on elementary tensors of the form $f\tp A$ with $f\in L^1(\Omega)$, $A\in X$. However, it seems that at certain crucial places, we need to genuinely consider Bochner integrable $X$-valued functions and not just the elementary tensors when making claims about a.e.~pointwise estimates; and elsewhere in the arguments, we do not seem to gain very much by only working with elementary tensors
%\end{thing}

Now let us make this procedure precise.
Since the eventual estimates we need for $\opconv$ depend on properties of the Bochner integral, we spend some time in this section on several small and routine results, to ensure we stay within the world of Bochner integrable functions during our construction.
 To reduce repetition we introduce the notation
\begin{equation}\label{eq:domain}
\cD \defeq \{ (r,s) \st r,s,r+s\in \Real^*\}.
\end{equation}
It is well known that when $(r,s)\in\cD$, $\pi_r\tp\pi_s$ is unitarily equivalent to a representation consisting of $\pi_{r+s}$ with infinite multiplicity.
% (A quick justification is that this is true when one restricts to the centre of $\bbH$, by an easy calculation with characters, and then one can invoke the Stone--von Neumann theorem.)
The unitaries which implement this equivalence occur in our explicit formula for $\opconv$, so we shall now state a more precise version of this intertwining result. It seems to be implicitly known, but we did not find part~\ref{li:measurable intertwiners} explicitly stated in the sources we consulted.

\begin{prop}[Fusion rules for $\bbH$, with continuity of intertwiners]
\label{pedantic fusion}
There exists a family of unitaries $(W_{r,s})_{(r,s)\in\cD}\subset \cU(L^2(\RR))$ with the following properties:
\begin{newnum}
\item\label{li:fusion}
 $\pi_r(\bx)\tp \pi_s(\bx) = W_{r,s}^* (\pi_{r+s}(\bx)\tp I)W_{r,s}$ for all $\bx\in \bbH$;
\item\label{li:measurable intertwiners} the functions $(r,s)\mapsto W_{r,s}$ and $(r,s)\mapsto W_{r,s}^*$ are both SOT-continuous functions $\cD\to \cU(L^2(\RR))$.
\end{newnum}
\end{prop}

\begin{proof}
For $(r,s)\in\cD$, define $W_{r,s}: L^2(\RR)\to L^2(\RR)$ by
\begin{equation}\label{eq:out-of-hat}
W_{r,s} F(h,k) = F\left( h-\frac{ks}{r+s} \,, h+k - \frac{ks}{r+s} \right)\qquad(F\in L^2(\RR), h,k\in \Real).
\end{equation}
A little thought shows $W_{r,s}$ is unitary. Using \eqref{eq:define schroedinger}, one may verify by direct calculation that
\[ W_{r,s}(\pi_r(\bx)\tp \pi_s(\bx)) (f\tp g) = (\pi_{r+s}(\bx)\tp I ) W_{r,s} (f\tp g) \]
for all $f,g\in L^2(\Real)$ and all $\bx\in \bbH$. So by density, $W_{r,s}$ intertwines $\pi_r\tp\pi_s$ with $\pi_{r+s}\otimes I$.

It remains to prove WOT-continuity, and hence, SOT-continuity, of the function $\cD\to\cU(L^2(\RR))$, $(r,s)\mapsto W_{r,s}$. This can be verified with a direct calculation and basic estimates. Alternatively, as pointed out by the referee of this article, there is a slicker approach.
Let $\gamma:\cD \to \SL(2,\Real)$ be the continuous function
\[ \gamma(r,s) = \twomat{\frac{r}{r+s}}{\frac{s}{r+s}}{-1}{1}\ . \]
With an obvious and harmless abuse of notation, we have
\[ W_{r,s} F\twoc{h}{k} = F \left(\gamma(r,s)^{-1}\twoc{h}{k}\right). \]
Let $\sigma:\SL(2,\Real) \to \cU(L^2(\RR))$ be defined by
$[\sigma(A)f]({\bf v}) = f(A^{-1}{\bf v})$. This is a WOT-continuous unitary representation of $\SL(2,\Real)$ (to be precise, it is the \dt{quasi-regular representation} arising from the subgroup $\SO(2,\Real)$). Then $\sigma\circ\gamma:\cD\to \cU(L^2(\RR))$ is WOT-continuous and $W_{r,s}=\sigma(\gamma(r,s))$ for all $(r,s)\in\cD$.
\end{proof}

\begin{rem}\label{r:out-of-hat}
For sake of brevity we omitted any explanation of how one arrives at the formula~\eqref{eq:out-of-hat}. In fact it can be derived as a special case of general results on tensoring induced representations, cf.~Remark~\ref{r:sch-as-induced}.
See Section~2.8 of \cite{KaniuthTaylor} for an accessible exposition of these techniques.
\end{rem}

\begin{lem}\label{measurability lemma}
Let $\cK$ be a separable Hilbert space, and let $\cS_1(\cK)$ be the space of trace-class operators on $\cK$.
Let $(\Omega,\mu)$ be a $\sigma$-finite measure space.
 Suppose $F:\Omega\to \cS_1(\cK)$ is measurable and $V:\Omega\to \Bdd(\cK)$ is weak-star measurable. Then the function $F\cdot V: \Omega\to\cS_1(\cK)$, $\omega\mapsto F(\omega)V(\omega)$, is measurable.
\end{lem}

\begin{proof}
Since the pointwise limit of a sequence of measurable functions is measurable, and since finite sums of measurable functions are measurable, it 
suffices to prove this result in the special case where $F=\chi_E\tp A$ for some measurable $E\subseteq\Omega$ and $A\in \cS_1(\cK)$.
Then, since the function $FV$ takes values in a \emph{separable} Banach space, it suffices by the Pettis measurability theorem
(see Theorem 2 in \cite[\S II.1]{DiestelUhl_book} or Theorem 1.1.1 in \cite{ABHN_book})
 to show that for each $B\in\Bdd(\cK)$ the function $\omega\mapsto \chi_E(\omega)\Tr(BAV(\omega))$ is measurable. But this is now obvious since we assumed $\omega\mapsto V(\omega)$ is weak-star measurable.
\end{proof}

\begin{lem}[Slicing with a trace in the second variable]
\label{trace and trace}
Let $\cH$ and $\cK$ be Hilbert spaces.
There is a well-defined, contractive linear map $I\tp\Tr: \cS_1(\cH \tp_2\cK) \to \cS_1(\cH)$ which sends $A\tp B$ to $\Tr(B)A$ whenever $A\in\cS_1(\cH), B\in\cS_1(\cK)$. Moreover, if $R\in\cS_1(\cH\tp_2\cK)$, 
let $R_1 = (I\tp \Tr)(R)$. Then for any $C\in \Bdd(\cH)$ we have $\Tr[CR_1] = \Tr[ (C\tp I)R]$.
\end{lem}

\begin{proof}
It suffices to prove that $I\tp\Tr : \cS_1(\cH)\ptp \cS_1(\cK)\to\cS_1(\cH)$ extends boundedly to the larger domain $\cS_1(\cH\tp_2\cK)$. Once this is done, the rest follows by checking the putative identities on suitable dense subspaces and extending by continuity. But the boundedness result is an easy consequence of the known identification $\cS_1(\cH\tp_2\cK)$ with $\cS_1(\cH)\optp\cS_1(\cK)$, where $\optp$ denotes the projective tensor product in the category of operator spaces and completely bounded maps.
(For an explanation and proof of this identification, see e.g.~\cite[Proposition 7.2.1]{effros-ruan-book}.)
\end{proof}

\begin{rem}\label{rem:slice by hand}
One can replace the appeal to operator-space techniques with a direct argument as follows.
Given index sets $\bbI$ and $\bbJ$ and vectors $u,v\in \ell^2(\bbI\times\bbJ)$, form the rank-one operator $u\tp v\in\cS_1(\ell^2(\bbI\times\bbJ))$. Then 
$(I\tp\Tr)(u\tp v)$ can be identified with $\sum_{j\in\bbJ} u_{\blob,j} \tp v_{\blob,j}$\/, and for each $j\in\bbJ$ the trace-class norm of $u_{\blob,j}\tp v_{\blob,j}$ in $\cS_1(\ell^2(\bbI))$ is bounded above by $\norm{u_{\blob,j}}_2 \norm{v_{\blob,j}}_2$. By Cauchy--Schwarz, $\sum_j \norm{u_{\blob,j}}_2 \norm{v_{\blob,j}}_2 \leq \norm{u}_2\norm{v}_2$\/, and so $(I\tp \Tr)(u\tp v)$ is trace class with control of the trace-class norm.
\end{rem}

We are now ready to give our description of $\opconv$. We proceed using two more lemmas.

\begin{lem}%\label{stage 1}
Let $F$ and $G$ be Bochner integrable functions $\Real\to \STN_1$. Then the function
\begin{equation}\label{eq:smeared}
 \theta_1( F \tp G)(r,s) = 
\begin{cases}
 (I\tp\Tr)[W_{r,s}(F(r)\tp G(s))W_{r,s}^* ] & \quad\text{for all $(r,s)\in\cD$,} \\
0 & \quad\text{otherwise},
\end{cases}
\end{equation}
is Bochner integrable. The map $\theta_1: L^1(\Real,\STN_1)\ptp L^1(\Real,\STN_1)\to L^1(\RR, \STN_1)$ is contractive and linear.
\end{lem}

\begin{proof}
Let $F$ and $G$ be Bochner integrable functions $\Real\to\STN_1$ (so, in particular, measurable functions).
Then $F\tp G$ is measurable when viewed as a map $\RR\to \STN_1\ptp\STN_1 \subset \cS_1(L^2(\RR))$.
Recall (see Remark~\ref{r:SOT-implies-wsmeas}) that SOT-continuous functions $\cD\to\Bdd(L^2(\RR))$ are weak-star measurable.  So combining Proposition~\ref{pedantic fusion} with two applications of Lemma \ref{measurability lemma} shows that
\[ (r,s) \mapsto W_{r,s}(F(r)\tp G(s))W_{r,s}^* \]
is measurable as a function $\cD\to \cS_1(L^2(\RR))$. Slicing with $I\tp \Tr$ we conclude that the right-hand side of \eqref{eq:smeared} is measurable. Moreover,
\[ \begin{aligned}
& \int_{\RR} \norm{  (I\tp\Tr)[W_{r,s}(F(r)\tp G(s))W_{r,s}^* ]}_1 \,d(r,s) \\
 \leq & \int_{\RR} \norm{F(r)}_1 \norm{G(s)}_1 \,d(r,s) = \norm{F}_{\LRS}\norm{G}_{\LRS}< \infty\,,
\end{aligned}
\]
and thus the right-hand side of \eqref{eq:smeared} is Bochner integrable.
The final claim about linearity and contractivity of $\theta_1$ is then routine book-keeping.
\end{proof}

The next lemma is a special case of a standard construction for vector-valued~$L^1$, but we include the statement explicitly for sake of clarity.
\begin{lem}[Vector-valued convolution]
Let 
$K:\RR\to \STN_1$, 
be Bochner integrable. Then
\[ \theta_2(K)(t) \defeq  \int_{\Real} K(p,t-p) \,dp \]
exists for a.e.~$t\in\Real$, and is Bochner integrable as a function $\Real\to\STN_1$.
Moreover,
\begin{equation}\label{eq:bloody house party}
  \norm{\theta_2(K)(t)}_1 \leq \int_{\Real} \norm{K(p,t-p)}_1 \,dp 
\quad\text{for a.e.\ $t\in\Real$},
\end{equation}
and
\[ \begin{aligned}
\norm{\theta_2(K)}_{\LRS}
& = \int_{\Real} \norm{\theta_2(K)(t)}_1  \,dt \\
& \leq \int_{\Real^2} \norm{K(p,t-p)}_1 \,d(p,t) = \norm{K}_{L^1(\RR,\STN_1)} \,,
\end{aligned}
\]
so that $\theta_2$ is a well-defined linear contraction $L^1(\RR,\STN_1)\to L^1(\Real,\STN_1)$.
\end{lem}

\begin{proof}
This follows by a standard application of the vector-valued Fubini theorem.
(For instance, apply~\cite[Theorem 1.1.9]{ABHN_book}.)
\end{proof}

\para{The explicit formula for $\opconv$}
Consider $\theta_2\theta_1: L^1(\Real,\STN_1)\ptp L^1(\Real,\STN_1)\to L^1(\Real,\STN_1)$. Now that we have taken care of all measurability issues, we can describe this as follows:
\begin{equation}\label{eq:explicit}
\theta_2\theta_1(F\tp G)(t) = \int_{\Real} (I\tp \Tr)[W_{r,t-r} (F(r)\tp G(t-r) ) W_{r,t-r}^*] \,dr \qquad\text{(a.e.\ $t\in\Real$)},
\end{equation}
where the integral on the right hand side of \eqref{eq:explicit} is a Bochner integral.

\begin{thm}\label{t:explicit}
$F\opconv G = \theta_2\theta_1(F\tp G)$ for all $F,G\in\LRS$.
\end{thm}

\begin{proof}
Let $F,G\in L^1(\Real,\STN_1)$ and put $K = \theta_1(F\tp G) \in L^1(\RR,\STN_1)$.
Since
$\theta_2(K)(t) = \int_{\Real} K(p,t-p)  \,dp$\/,
we have 
\[ \Tr\left[ \theta_2(K)(t) \pi_t(x)^* \right] = \int_{\Real} \Tr\left[ K(p,t-p)\pi_t(\bx)^* \right]\,dp \]
and so, using Fubini's theorem (scalar-valued case),
\[ \begin{aligned}
\Psi\cT\theta_2(K)(\bx)
& = \int_{\Real} \left( \int_{\Real} \Tr
\left[ K(p,t-p)\pi_t(\bx)^* \right]\,dp \right)  \,dt \\
& = \int_{\RR}  \Tr
\left[ K(r,s)\pi_{r+s}(\bx)^* \right]\,d(r,s) \,.
\end{aligned} \]
For a.e.~$(r,s)\in\cD$, we have
\[ \begin{aligned}
& \phantom{=}    \Tr\left[ K(r,s)\pi_{r+s}(\bx)^* \right] \\
& =  \Tr \left[ (\pi_{r+s}(\bx)^*\tp I) W_{r,s}(F(r)\tp G(s))W_{r,s}^*\right]  & \quad\text{(by Eq.\ \eqref{eq:smeared} and Lem.\ \ref{trace and trace})} \\
& =  \Tr \left[ W_{r,s}^*(\pi_{r+s}(\bx)^*\tp I) W_{r,s}(F(r)\tp G(s))\right]  \\
& =  \Tr \left[ (\pi_r(\bx)^*\tp \pi_s(\bx)^*)(F(r)\tp G(s))\right] & \quad\text{(by Prop.\ \ref{pedantic fusion})} \\
& =  \Tr [ F(r)\pi_r(\bx)^*] \Tr [G(s)\pi_s(\bx)^* ]\,.
\end{aligned} \]
Therefore,
\[
 \begin{aligned}
\Psi\cT\theta_2(K)(\bx)
& = \int_{\RR}    \Tr [ F(r)\pi_r(\bx)^*] \Tr [G(s)\pi_s(\bx)^* ] \,d(r,s) \\
&  = \Psi\cT(F)(\bx) \Psi\cT(G)(\bx) \\
& = \Psi\cT(F\opconv G)(\bx)\quad\text{for all $\bx\in\bbH$},
\end{aligned}
\]
so that $F\opconv G = \theta_2(K)=\theta_2\theta_1(F\tp G)$, as required.
\end{proof}

\begin{cor}\label{key inequality}
$\norm{F\opconv G}_{\LIRS} \leq \norm{F}_{\LRS} \norm{G}_{\LIRS}$ for all $F,G\in\LRS$.
\end{cor}

\begin{proof}
Let $K=\theta_1(F\tp G)\in L^1(\RR,\STN_1)$. Then
$\norm{K(r,s)}_1\leq \norm{F(r)}_1\norm{G(s)}_1$ for a.e.\ $(r,s)\in\Real^2$. So, using the norm bound \eqref{eq:bloody house party}, we have
\[
\norm{\theta_2(K)(t)}_1
 \leq \int_{\Real} \norm{F(p)}_1\norm{G(t-p)}_1  \,dp 
 \leq \norm{F}_{\LRS} \norm{G}_{\LIRS} \quad\text{for a.e.\ $t\in\Real$}.
\]
Hence $\norm{F\opconv G}_{\LIRS} = \norm{\theta_2(K)}_{\LIRS} \leq \norm{F}_{\LRS}\norm{G}_{\LIRS}$ as required.
\end{proof}

It would be interesting to study similar concrete realizations of ``dual convolution'' for other Type~I groups where we know the fusion rules explicitly, in particular for the Euclidean motion group~$\Euc(2)$.
 Similar arguments to those used for $\FAH$ yield $\FA(\Euc(2))\iso L^1(\Real_+,\cS_1(L^2(\mathbb T)))$ (see \cite[Chapter~4]{Sug_book_ed2}),
 but the ``dual convolution'' is then governed by a certain hypergroup structure on $\Real_+$, rather than the group structure on $\Real$ which governs dual convolution for~$\FAH$.  We leave a more detailed look at this case for future work.

Let us finish this section by completing the proof of Theorem~\ref{t:W-is-a-module-norm}.

\begin{proof}[Proof of Theorem \ref{t:W-is-a-module-norm}]
It suffices to prove the inequality \eqref{eq:W-module-norm} -- the other statements in the theorem follow easily.
Thus, fix $f\in\FAH$ and $h\in \sW_0$ (so $h$ and $fh$ are integrable). To prove that $\wnorm{hf}\leq \wnorm{h}\norm{f}_{\FAH}$, we follow the idea outlined at the end of Section~\ref{s:define derivation}, but for technical convenience we work on the Fourier side rather than the group side.

Let $G:\Real\to \STN_1$ be a measurable function with compact support, taking values a.e.~in the unit ball of $\STN_1$. In particular, $\norm{G}_{\LIRS}\leq 1$ and $\Psi\cT(G)\in \FA(\bbH)$.
Using the ``adjoint relation'' \eqref{eq:adjoint},
\begin{equation}\label{eq:control-hf}
 \int_{\bbH} \check{h}(\bx) \check{f}(\bx) \Psi\cT(G)(\bx) \,d\bx = \int_{\Real^*} \Tr(\pi_t(hf)G(t))\,dt \,.
\end{equation}
On the other hand, since $\check{f}\Psi\cT(G)\in\FA(\bbH)$ it equals $\Psi\cT(F)$ for some unique $F\in \LRS$,
and applying \eqref{eq:adjoint} again gives
\begin{equation}\label{eq:use-fg}  \int_{\bbH} \check{h}(\bx) \Psi\cT(F)(\bx) \,d\bx = \int_{\Real^*} \Tr(\pi_t(h)F(t))\,dt \,.
\end{equation}
As $F =(\Psi\cT)^{-1}(\check{f}) \opconv G$, using Corollary~\ref{key inequality} and the fact $\Psi\cT$ is an isometry gives
\[ \norm{F}_{\LIRS} \leq \norm{\check{f}}_{\FAH}  = \norm{f}_{\FAH} \,. \]
Combining this with \eqref{eq:control-hf} and \eqref{eq:use-fg} yields
\begin{equation}\label{eq:almost}
\begin{aligned}
\left\vert \int_{\Real^*} \Tr(\pi_t(hf)G(t))\,dt  \right\vert
& = \left\vert \int_{\Real^*} \Tr(\pi_t(h)F(t))\,dt \right\vert \\
& \leq \norm{\pi_\bullet(h)}_{\LRB} \norm{F}_{\LIRS} \\
& \leq \norm{\pi_\bullet(h)}_{\LRB}\norm{f}_{\FAH}   \,.
\end{aligned}
\end{equation}
Let $S$ denote the supremum on the left hand side of \eqref{eq:almost} over all such $G$. Since $(\STN_1)^*=\sB$ isometrically, and since $t\mapsto \pi_t(fh)$ is Bochner integrable (see Section~\ref{ss:vector valued}), a straightforward approximation argument with simple functions yields $S=\norm{\pi_\bullet(hf)}_{\LRB}$.
We conclude that $\wnorm{hf}\leq \wnorm{h}\norm{f}_{\FAH}$ as required.
\end{proof}

\begin{rem}[Remarks on the proof]\label{r:excuses}\
\begin{newnum}
\item The idea guiding our duality argument is, of course, that $\LIRS$ is isometric to a subspace of $L^1(\Real,\sB)^*$, and that the intersection of the unit ball of $\LIRS$ with $\LRS$ is a norming subset for this pairing. However, we wished to avoid technical discussions about the Radon--Nikodym property and duality for vector-valued $L^p$-spaces; cf.~Theorem~1 in \cite[Chapter~4]{DiestelUhl_book}.
\item One might wish to bypass the use of the ``adjoint relation'' and have a direct proof that $\norm{F\opconv G}_{\LRB}\leq \norm{F}_{\LRS} \norm{G}_{\LRB}$. The natural attempt is to consider\hfill\break
$\int_{\Real} \Tr((F\opconv G)(t)A_t)\,dt$
where $A_\bullet\in \LIRS$. Now we have
\[ \Tr((F\opconv G)(t)A_t) = \int_{\Real} \Tr\left((A_t\tp I)W_{r,t-r}(F(r)\tp G(t-r))W_{r,t-r}^*\right)\,dr \qquad\text{(a.e.~$t\in\Real$)}; \]
but trying to get upper bounds on the right-hand side with crude tools is problematic, since $A_t\tp I$ is not trace class and since conjugation with  $W_{r,t-r}$ will not preserve $\STN_1\ptp \sB$.
\end{newnum}
\end{rem}

\begin{rem}[Fourier coefficients of pointwise products]
Following a suggestion of the referee, we note that Theorem~\ref{t:explicit} may be viewed as an expression for the Fourier coefficients of the product of two functions in $(\FA\cap L^1)(\bbH)$, in terms of a twisted convolution of their Fourier series. This works as follows: let $f_1, f_2\in\FAH\cap L^1(\bbH)$; then by Theorem~\ref{t:F-inversion-for-AH}, $|t| \pi_t(f_i)\in\STN_1$ for a.e.~$t\in\Real$ and the functions $F_i(t)=|t|\pi_t(f_i)$ belong to $\LRS$, for $i=1,2$. Since $\Psi\cT(F_1\opconv F_2)=\Psi\cT(F_1)\Psi\cT(F_2)=f_1f_2$, another application of Theorem~\ref{t:F-inversion-for-AH} gives
\begin{equation}\label{eq:referee-prefers}
|t| \pi_t(f_1f_2) = (F_1\opconv F_2)(t) = \int_{\Real} (I\tp\Tr)[W_{r,t-r}(|r|\pi_r(f_1)\tp |t-r|\pi_{t-r}(f_2))W_{r,t-r}^*]\, dr\,.
\end{equation}
Using suitable regularization arguments, one can show that Equation~\eqref{eq:referee-prefers} remains valid for more general $f_1$ and $f_2$\/, but we shall not pursue this topic here.
\end{rem}

\end{section}

\begin{section}{Extending our result to other Lie groups}\label{s:moreLie}

We recall some definitions. Our indexing conventions are those of~\cite{HN_LieBook}.
\begin{dfn}
Let $\fg$ be a non-zero Lie algebra.
The \dt{lower central series} of $\fg$ is the decreasing sequence of ideals
\[ \fg = C_1 \supseteq C_2 \supseteq \dots \]
where $C_{j+1}\defeq [\fg, C_j]$ for each $j\geq 1$. We say $\fg$ is \dt{nilpotent} if $C_{n+1}=\{0\}$ for some~$n$ (note that this forces $C_{j+1}$ to be a \emph{proper} subset of $C_j$ for each $0\leq j \leq n$). The least such $n$ is called the \dt{nilpotency degree} of $\fg$; if $\fg$ has nilpotency degree $d$, we say that $\fg$ is \dt{$d$-step nilpotent.}
\end{dfn}

Our first task in this section is to prove Theorem~\ref{t:nilpotent cases}, so let us remind ourselves what it says.

\para{\thf Theorem \ref{t:nilpotent cases} {\rm (reprise)}}
Let $G$ be a $1$-connected, nilpotent, non-abelian Lie group. Then $\FA(G)$ is not weakly amenable.

\medskip

Define $\fh_3$ to be the real Lie algebra spanned as a real vector space by elements $x$, $y$ and $[x,y]$ satisfying the relations $[x,[x,y]]=[y,[x,y]]=0$. It is the Lie algebra of the group $\bbH$.

\begin{lem}\label{l:copy of heisenberg}
If $\fg$ is non-abelian and nilpotent, it contains a copy of~$\fh_3$.
\end{lem}

\begin{proof}
Let $C^{n-1}\supset C^n\supset C^{n+1}=\{0\}$ be the last terms in the lower central series of $\fg$, so that $C^n=Z(\fg)$. Pick $x\in C^{n-1}\setminus C^n$, and pick $y\in\fg$ such that $[x,y]\neq 0$. Then, since $[x,y] \in Z(\fg)$, we see that $\Real x+\Real y +\Real[x,y]$ is a Lie subalgebra of $\fg$ which is isomorphic to~$\fh_3$\/.
\end{proof}

%The following is \cite[Corollary 1.111]{Knapp_beyond-ed2}:
%\begin{quote}
%if $N$ is a simply connected analytic group then any analytic subgroup of $N$ is simply connected and closed.
%\end{quote}

\begin{proof}[Proof of Theorem \ref{t:nilpotent cases}]
By Proposition~\ref{p:WA-hered}, it suffices to prove that $G$ contains a closed subgroup isomorphic to~$\bbH$.
This may well be folklore for Lie theorists, but we give the details for the reader's convenience.

Let $\fg$ be the Lie algebra of $G$; by Lemma~\ref{l:copy of heisenberg} there is a subalgebra $\fh\subseteq \fg$ which is isomorphic as a Lie algebra to $\fh_3$. Now since $G$ is a $1$-connected nilpotent Lie group, the exponential map $\exp_{\fg}$ of the Lie algebra $\fg$ maps $\fg$ \emph{diffeomorphically} onto $G$ (see, e.g.~Theorem 11.2.10 in \cite{HN_LieBook}), and therefore the image of $\fh$ under $\exp_{\fg}$  is a \emph{closed}, $1$-connected subgroup $H\subseteq G$. It remains to note that since $H$ and $\bbH$ are both $1$-connected, and their respective Lie algebras $\fh$ and $\fh_3$ are isomorphic, the two groups are isomorphic as Lie groups.
\end{proof}

It would be highly desirable to remove the condition of simple-connectedness, but we have been unable to do this.
% However, it seems likely that a more careful look at the stucture theory of nilpotent Lie groups would yield the desired generalization.

Finally, we make some brief comments on the solvable cases.
(Recall from \cite[Theorem~5.5]{CG_WAAG1} that if $G$ is a $1$-connected, simply connected Lie group which is not solvable, then $\FA(G)$ is not weakly amenable.)
We start by quoting without proof a result from the theory of Lie algebras.

\begin{lem}[{\cite[Corollary 5.4.12]{HN_LieBook}}]
Let $\fg$ be a finite-dimensional, solvable Lie algebra. Then the commutator ideal $[\fg,\fg]$ is nilpotent.
\end{lem}

If $G$ is a $1$-connected Lie group with Lie algebra $\fg$, and $\fk$ is an ideal in $\fg$, then the subgroup of $G$ corresponding to $\fk$ is closed and has $\fk$ as its Lie algebra. Moreover, in this setting $[\fg,\fg]$ is the Lie algebra of the derived subgroup $[G,G]$.
% found this as an exercise in Fulton--Harris
Combined with Lemma~\ref{l:copy of heisenberg}, this leads to the following result, no doubt well-known to specialists.

\begin{prop}
Let $G$ be a $1$-connected solvable Lie group, of solvable length $\geq 3$. Then $G$ contains a closed subgroup isomorphic to $\bbH$.
\end{prop}

Therefore, for such $G$, $\FA(G)$ is not weakly amenable. Summing up, and appealing again to \cite[Theorem 5.5]{CG_WAAG1}, we can state the following result:

\begin{thm}
Let $G$ be a $1$-connected Lie group whose Fourier algebra is weakly amenable. Then either $G$ is abelian, or it is $2$-step solvable; and it contains no closed copy of $\bbH$ or the real $ax+b$ group.
\end{thm}

These observations suggest the following question,  which is left for future work.

\para{Question 1}
Let $\Euc(2)= \Real^2 \rtimes \SO(2)$ be the Euclidean motion group, and let $\widetilde{\Euc}(2)$ be its universal cover. Are either $\FA(\Euc(2))$ or $\FA(\widetilde{\Euc}(2))$ weakly amenable?

\bigskip
We close with a question motivated by general concepts in the study of derivations on Banach algebras.
By standard arguments, once we have shown $D:\FAH\to\sW$ is a non-zero continuous derivation,
we can now obtain non-zero continuous derivations from $\FAH$ to its dual. To be specific: pick any $\psi\in \sW^*$, and define $D_\psi: \FAH\to\FAH^*$ by
$D_\psi(f)(g) \defeq \psi( D(f)\cdot g)$ for each $f,g\in\FAH$.
It is easily checked that $D_\psi$ is a continuous derivation; and since $D$ is not identically zero, it is clear that we can find $\psi\in\sW$ such that $D_\psi$ is not identically zero.

However, unlike the derivations that have been constructed on all previous examples --- that is, on Fourier algebras $\FA(G)$ where $G$ is either compact or one of the groups from \cite{CG_WAAG1} --- our derivations $D_\psi$ are in general not \emph{cyclic derivations}. (See the early sections of \cite{CG_WAAG1} for a discussion of cyclic derivations and cyclic amenability for commutative Banach algebras.)

\para{Question 2}
Is $\FAH$ cyclically amenable? Equivalently: does there exist a non-zero continuous derivation $T:\FAH\to\FAH^*$ which satisfies $T(a)(b)+T(b)(a)=0$ for all $a,b\in\FAH$?

\end{section}

%%% Local Variables:
%%% TeX-master: "WA_hberg_arX2"
%%% End:

%%%%%%%%%%%%%%%%%

\bibliographystyle{siam}
\bibliography{WAAHbib}

\vfill

\para{Affiliations}\

\noindent
Y. Choi\\
Department of Mathematics and Statistics\\
Fylde College, Lancaster University\\
Lancaster, United Kingdom LA1 4YF

\noindent
Email: \texttt{y.choi1@lancaster.ac.uk}

\bigskip
\noindent
M. Ghandehari\\
Department of Pure Mathematics\\
University of Waterloo\\
200 University Avenue West\\
Waterloo (ON), Canada N2L 3G1

\noindent
Email: \texttt{mghandehari@uwaterloo.ca}

\end{document}